\documentclass[12pt]{amsart}
\usepackage{amsfonts}
\usepackage{mathrsfs}
\usepackage{amsmath}
\usepackage{amsmath, amsthm, amssymb}
\usepackage{color}
\numberwithin{equation}{section}
\renewcommand{\theequation}{\theequation. \arabic{equation}}
\numberwithin{equation}{section}
\newtheorem{thm}{Theorem}[section]

\newtheorem{rem}[thm]{Remark}
\newtheorem{prop}[thm]{Proposition}
\newtheorem{defn}[thm]{Definition}

\def\squarebox#1{\hbox to #1{\hfill\vbox to #1{\vfill}}}

\begin{document}
\title[On the complex Hermite polynomials]
{On  the complex Hermite polynomials and partial differential equations}
%\dedicatory{Dedicated to Professor Donald Knuth on %the occasion of his $100$th birthday}
\dedicatory{This paper is dedicated to the memory of Professor Atiyah and Dr.Lewis}
\author{Zhi-Guo Liu}
\date{\today}
\address{School of Mathematical Sciences and  Shanghai  Key Laboratory of PMMP, East China Normal University, 500 Dongchuan Road,
	Shanghai 200241, P. R. China} \email{zgliu@math.ecnu.edu.cn;
	liuzg@hotmail.com}

\thanks{This work was supported by the National Science Foundation of China (Grant No. 11571114) and Science and Technology Commission of Shanghai Municipality (Grant No. 13dz2260400).}

\begin{abstract}
In this paper we use a set of  partial differential equations to prove an expansion theorem for multiple complex Hermite polynomials. This expansion theorem allows us to develop a systematic and completely new approach  to  the complex Hermite polynomials.
Using this expansion, we derive the Poisson Kernel,  the Nielsen type formula, the addition formula for the complex Hermite polynomials with ease.  A multilinear generating function for the complex Hermite polynomials is proved.
\end{abstract}
%\dedicatory{Dedicated to Professor George E. Andrews on the occasion  of his $80$th birthday}
\maketitle
%%%%%%%%%%%%%%%%%%%%%%%%%%%%%%%%%%%%%%%%%%%%%%%%%%%%%%%%%%%%%%%%%%%%%%%%%%%%%%%%%%
\section{Introduction and preliminary}
%%%%%%%%%%%%%%%%%%%%%%%%%%%%%%%%%%%%%%%%%%%%%%%%%%%%%
\noindent
 With the aid of  a system of partial differential equations,  we proved an expansion theorem for the bivariate Hermite polynomials in \cite[Theorem~1.8]{LiuHermite2017}. This expansion theorem allows us to develop a systematic method to prove the identities involving the Hermite polynomials.  I find the idea of \cite{LiuHermite2017}  has universal significance, which stimulates us to develop a new method to treat  the complex Hermite polynomials.

\begin{defn}\label{comHermite}
For complex numbers $x, y$ and non-negative integers $m, n$, the complex Hermite polynomials are defined by
\begin{equation*}
H_{m, n}(x, y)=\sum_{k=0}^{m\land n}(-1)^k k! {m\choose k} {n\choose k}
x^{m-k} y^{n-k},
\end{equation*}
where $m\land n=\min\{m, n\}.$
\end{defn}
The polynomials $H_{m, n}(z, \bar{z})$ were first considered by It\^{o} \cite{Ito1952} in his study of complex multiple Wiener integrals and their applications to normal stochastic processes.  These polynomials are also applied in \cite{AliBH2019} to coherent states, and in \cite{Wunsche1998}, \cite{Wunsche1999} to quantum optics and quasi-probabilities respectively.  Several papers about this topic have been published in recent year,  see for example  \cite{Ghanmi2013}, \cite{Ismail2016}, \cite{IsmailZhang}, \cite{IsmailZeng2015}.

For our purpose, we need extend slightly the complex Hermite polynomials by adding an extra parameter to them, and for convenience, we still call the extended complex Hermite polynomials as the complex Hermite polynomials.
\begin{defn}\label{triHermitedefn} For any complex numbers $x, y$ and $z$, the complex Hermite polynomials $H_{m, n}(x, y, z)$ are defined as
	\begin{equation*}
	H_{m, n}(x, y, z)=\sum_{k=0}^{m\land n} k! {m\choose k} {n\choose k}
	x^{m-k} y^{n-k} z^k.
	\end{equation*}
\end{defn}
It is obvious that when $z=-1,$ $H_{m, n}(x, y, z)$ reduce to the usual complex Hermite polynomials $H_{m, n}(x, y).$  By a simple calculation, we also find the following proposition.
\begin{prop}\label{Hermiterelation}
The polynomials $H_{m, n}(x,y,z)$ and the polynomials $H_{m, n}(x, y)$ satisfy
\[
H_{m, n}(x, y, z)=\left( \sqrt{-z}\right)^{m+n} H_{m,n} \left(\frac{x}{\sqrt{-z}}, \frac{y}{\sqrt{-z}}\right).
\]
\end{prop}
Thus we may regard $H_{m, n}(x, y, z)$ as a variant form of the usual complex Hermite polynomials $H_{m, n}(x, y)$.  Although $H_{m, n}(x, y, z)$ are equivalent to the complex Hermite polynomials $H_{m, n}(x, y)$, the former have a richer mathematical structure than the latter.
\begin{rem} \rm The polynomials  $H_{m, n}(x, y, -z)$ have been considered by Datolli et al. \cite[pp.23--24]{ DattoliOTTOVA1997}, and several basic properties about $H_{m, n}(x, y, -z)$  were obtained by them.
\end{rem}	
To state our expansion theorem, we now introduce the definition of the $k$-fold
complex Hermite series in several variables.
\begin{defn}\label{kfoldHermite:pp} The $k$-fold complex Hermite series are defined as
	\[
	\sum_{m_1, n_1, \ldots, m_k, n_k=0}^\infty \lambda_{m_1, n_1, \ldots, m_k, n_k} H_{m_1, n_1} (x_1, y_1, z_1) \cdots H_{m_k, n_k} (x_k, y_k, z_k),
	\]	
	where $\lambda_{m_1, n_1, \ldots, m_k, n_k}$ are complex numbers independent of
	$x_1, y_1, z_1, \ldots, x_k, y_k, z_k.$
\end{defn}	
The principal result of this paper is the following expansion theorem
for the analytic functions in several variables.
\begin{thm}\label{LiutriHermite:eqna}
	If $f(x_1, y_1, z_1, \ldots, x_k, y_k, z_k)$ is a $3k$-variable analytic function at $(0, 0, \ldots, 0)\in \mathbb{C}^{3k}$, then, $f$ can be expanded in an absolutely and uniformly convergent $k$-fold  complex  Hermite series, if and only if, for $j\in \{1, 2, \ldots, k\}, f$ satisfies the partial differential equations
	\[
	\frac{\partial f}{\partial z_j}
	=\frac{\partial^2 f}{\partial x_j \partial y_j}.
	\]
\end{thm}	
This theorem is a powerful tool for proving formulas involving the  complex Hermite polynomials, which allows us to develop a systematic
method to derive identities involving the complex Hermite polynomials.
%%%%%%%%%%%%%%%%%%%%%%%%%%%%%%%%%%%%%%%%%%%%%%%%%%%%%%%%%%%%%%%%%%%%%%%%%%%%%%%%%%%%%
\section{The proof of Theorem~\ref{LiutriHermite:eqna}}
%%%%%%%%%%%%%%%%%%%%%%%%%%%%%%%%%%%%%%%%%%%%%%%%%%%%
 Using $\exp(sx+ty+stz)=\exp(sx)\exp(ty)\exp(stz)$ and  the Maclaurin expansion for the exponential function, one can easily derive  Proposition~\ref{triH:propa}.
\begin{prop}\label{triH:propa} For any complex numbers $x, y, z$ and $s, t$, we have
\[
\sum_{m, n=0}^{\infty} H_{m, n} (x, y, z) \frac{s^m t^n}{m! n!}
=\exp(sx+ty+stz).
\]	
\end{prop}	
In order to prove Theorem~\ref{LiutriHermite:eqna}, we need the  following three propositions.
\begin{prop}\label{triH:propb} The complex Hermite polynomials $H_{m, n}(x, y, z)$ satisfy the partial differential equation
\[
\frac{\partial H_{m, n}}{\partial z}
=\frac{\partial^2 H_{m, n}}{\partial x \partial y}.
\]	
\end{prop}
\begin{proof}
Applying the partial differential operator ${\partial^2}/{\partial x \partial y}$ to act both sides of the equation in Proposition~\ref{triH:propa}, we find that
\[
\sum_{m, n=0}^{\infty} \frac{\partial^2 H_{m, n}}{\partial x \partial y} \frac{s^m t^n}{m! n!}
=st\exp(sx+ty+stz).
\]
Upon differentiating both sides of the equation in Proposition~\ref{triH:propb} with respect to $z$, we arrive at
\[
\sum_{m, n=0}^{\infty} \frac{\partial H_{m, n}}{\partial z} \frac{s^m t^n}{m! n!}=st\exp(sx+ty+stz).
\]
A comparison of these two equations immediately gives us that
\[
\sum_{m, n=0}^{\infty} \frac{\partial H_{m, n}}{\partial z} \frac{s^m t^n}{m! n!}=\sum_{m, n=0}^{\infty} \frac{\partial^2 H_{m, n}}{\partial x \partial y} \frac{s^m t^n}{m! n!}.
\]
Equating the coefficients of like powers of $s$ and $t$, we complete the proof of the proposition.
\end{proof}

\begin{prop}\label{triH:propc} The following exponential operator representation for the complex Hermite polynomials holds:
\[
H_{m, n} (x, y, z)
=\exp \left( z\frac{\partial^2 }{\partial x \partial y} \right)\{x^m y^n\}.
\]
\end{prop}
This operational identity for the complex Hermite polynomials is equivalent to  \cite[Eq.(1.5.2d)]{DattoliOTTOVA1997}.
%%%%%%%%%%%%%%%%%%%%%%%%%%%%%%%%%%%%%%%%%%%%%%%%
\begin{rem}\rm
Using the exponential operator
$\exp \left( -z\frac{\partial^2 }{\partial x \partial y} \right)$	
to act both sides of the equation in Proposition~\ref{triH:propc}, we
have
\begin{align}
x^my^n&=\exp \left( -z\frac{\partial^2 }{\partial x \partial y} \right)\{H_{m, n}(x, y)\}\\
&=\sum_{k=0}^{m\land n} k! {m\choose k} {n\choose k}
\exp \left( -z\frac{\partial^2 }{\partial x \partial y} \right)\{x^{m-k} y^{n-k}\} z^k\nonumber\\
&=\sum_{k=0}^{m\land n} k! {m\choose k} {n\choose k}
H_{m-k}(x, y, -z) H_{n-k}(x, y, -z) z^k.\nonumber
\end{align}	
\end{rem}
\begin{prop}\label{mcvarapp}
If $f(x_1, x_2, \ldots,  x_k)$ is analytic at the origin $(0, 0, \ldots,  0)\in \mathbb{C}^k$, then,
$f$ can be expanded in an absolutely and uniformly convergent power series,
 \[
 f(x_1, x_2, \ldots,  x_k)=\sum_{n_1, n_2, \ldots, n_k=0}^\infty \lambda_{n_1, n_2, \ldots, n_k}
 x_1^{n_1} x_2^{n_2}\cdots x_k^{n_k}.
 \]
 \end{prop}
This proposition can be found in the standard textbooks for complex analysis in several variables (see, for example \cite[p. 5, Proposition~ 1]{Malgrange}).

 Now we begin to prove Theorem~\ref{LiutriHermite:eqna} with the help of the above three propositions.
\begin{proof} The theorem can be proved  by mathematical induction. We first prove the theorem for the case $k=1.$
	
Since $f$ is analytic at $(0, 0, 0),$  we know that $f$ can be expanded in
an absolutely and uniformly convergent power series in a neighborhood of $(0, 0)$. Thus there exists a sequence $\{\lambda_{m, n, p}\}$
independent of $x_1, y_1$ and $z_1$ such that
\begin{equation}
 f(x_1, y_1, z_1)=\sum_{m, n, p=0}^\infty \lambda_{m, n, p} x_1^m y_1^n z_1^p.
\label{liu:eqn1}
\end{equation}
The series on the right-hand side of the equation above is absolutely and uniformly convergent.

Upon substituting the equation above into the following partial differential equation:

\[
\frac{\partial f}{\partial z_1}
=\frac{\partial^2 f}{\partial x_1 \partial y_1},
\]
and then using the identities, $D_{ z_1}\{z_1^p\}=p z_1^{p-1}$, in the resulting equation,   we obtain
\begin{equation*}
\sum_{m, n, p=0}^\infty p \lambda_{m, n, p} x_1^{m} y_1^{n} z_1^{p-1}
=\frac{\partial^2 }{\partial x_1 \partial y_1}\left\{\sum_{m, n, p=0}^\infty \lambda_{m, n, p}  x_1^m y_1^{n}z_1^p\right\}.
\end{equation*}
Upon equating the coefficients of $z_1^{p-1}$ on both sides of the  equation, we deduce that
\begin{align*}
p\sum_{m, n=0}^\infty \lambda_{m, n, p} x_1^m y_1^n
=\frac{\partial^2 }{\partial x_1 \partial y_1}\left\{\sum_{m, n=0}^\infty \lambda_{m, n,  p-1} x_1^m y_1^n\right\}.
\end{align*}
If we iterate this relation $(p-1)$ times and
interchange the order of differentiation and summation,  we deduce that
\begin{align*}
\sum_{m, n=0}^\infty \lambda_{m, n, p} x_1^m y_1^n
&=\frac{1}{p!}\frac{\partial^{2p} }{\partial x_1^p \partial y_1^p}\left\{\sum_{m, n=0}^\infty \lambda_{m, n,  0} x_1^m y_1^n\right\}\\
&=\frac{1}{p!}\sum_{m, n=0}^\infty \lambda_{m, n,  0}
\frac{\partial^{2p} }{\partial x_1^p \partial y_1^p} \{x_1^m y_1^n\}.
\end{align*}
Substituting this equation into (\ref{liu:eqn1}) and using a simple calculation, we conclude that
\begin{align*}
 f(x_1, y_1, z_1)&=\sum_{ p=0}^\infty z^p \sum_{m. n=0}^\infty \lambda_{m, n, p} x_1^m y_1^n\\
 &=\sum_{ p=0}^\infty \frac{z_1^p}{p!}\sum_{m, n=0}^\infty \lambda_{m, n, 0}
 \frac{\partial^{2p} }{\partial x_1^p \partial y_1^p} \{x_1^m y_1^n\} .
\end{align*}
Interchanging the order of summation and using Proposition~\ref{triH:propc}, we deduce that
\begin{align*}
 f(x_1, y_1, z_1)&=\sum_{m, n=0}^\infty \lambda_{m, n, 0} \exp \left(z_1 \frac{\partial^2 }{\partial x_1, \partial y_1}\right) \{x_1^m y_1^n\}\\
 &=\sum_{m, n=0}^\infty \lambda_{m, n, 0} H_{m, n}(x_1, y_1, z_1).
\end{align*}
This indicates that $f(x_1, y_1, z_1)$ can be expanded in terms of $H_{m, n}(x_1, y_1, z_1).$

Conversely, if $f(x_1, y_1, z_1)$ can be expanded in terms of $H_{m, n}(x_1, y_1, z_1)$, then,  using Proposition~\ref{triH:propb},  we find that
$f(x_1, y_1, z_1)$ satisfies the partial differential equation
\[
\frac{\partial f}{\partial z_1}
=\frac{\partial^2 f}{\partial  x_1 \partial  y_1}.
\]
This shows  that Theorem~\ref{LiutriHermite:eqna} holds for the case with $k=1$.

Now, we assume that the theorem is true for the case $k-1$ and consider the case $k$.
If we regard $f(x_1, y_1, z_1, \ldots, x_k, y_k, z_k)$ as a function of $x_1, y_1 $ and $z_1$,  then,  $f$ is analytic at $(0, 0, 0)$ and satisfies
the partial differential equation
\[
\frac{\partial f}{\partial z_1}
=\frac{\partial^2 f}{\partial  x_1 \partial  y_1}.
\]
Hence there exists a sequence
$\{c_{m_1, n_1}(x_2, y_2, z_2,  \ldots, x_k, y_k, z_k)\}$ independent of $x_1, y_1$ and $z_1$ such that
\begin{align}
&f(x_1, y_1, z_1, \ldots, x_k, y_k, z_k)\label{liu:eqn2}\\
&=\sum_{m_1, n_1=0}^\infty c_{m_1, n_1}(x_2, y_2, z_2,  \ldots, x_k, y_k, z_k)H_{m_1, n_1}(x_1, y_1, z_1). \nonumber
\end{align}
Setting $z_1=0$ in the equation and using the obvious equation $H_{m_1, n_1}(x_1, y_1, 0)=x_1^{m_1} y_1^{n_1},$  we obtain
\begin{align*}
&f(x_1, y_1, 0, \ldots, x_k, y_k, z_k)\\
&=\sum_{m_1, n_1=0}^\infty c_{m_1, n_1}(x_2, y_2, z_2,  \ldots, x_k, y_k, z_k)x_1^{m_1} y_1^{n_1}.
\end{align*}
Using the Maclaurin expansion for analytic functions of  two variables, we immediately deduce that
\begin{align*}
&c_{m_1, n_1}(x_2, y_2, z_2 \ldots, x_k, y_k, z_k)\\
&=\frac{\partial^{m_1+n_1} f(x_1, y_1,  0,   \ldots, x_k, y_k, z_k)}{m_1!  n_1! \partial {x_1}^{m_1} \partial {y_1}^{n_1}}\Big|_{x_1=y_1=0}.
\end{align*}
Since $f(x_1, y_1, z_1, \ldots, x_k, y_k, z_k)$ is analytic at $(0, \ldots, 0)\in \mathbb{C}^{2k},$  from
the above equation, we know that $c_{m_1, n_1}(x_2, y_2, z_2,\ldots, x_k, y_k, z_k)$ is analytic at
\[
 (x_2, y_2, z_2, \ldots, x_k, y_k, z_k)=(0, \ldots, 0)\in \mathbb{C}^{3k-3}.
\]
Substituting (\ref{liu:eqn2}) into the partial differential equations in Theorem~\ref{LiutriHermite:eqna}, we find that for $j=2, \ldots,  k,$
\begin{align*}
&\sum_{n_1=0}^\infty  \frac{\partial c_{m_1, n_1}(x_2, y_2, z_2, \ldots, x_k, y_k, z_k)} {\partial {z_j}}H_{m_1, n_1}(x_1, y_1, z_1)\\
&=\sum_{m_1, n_1=0}^\infty \frac{\partial ^2 c_{m_1, n_1}(x_2, y_2, z_2, \ldots, x_k, y_k, z_k)}{\partial {x_j} \partial {y_j}} H_{n_1}(x_1, y_1, z_1).
\end{align*}
By equating the coefficients of $H_{m_1, n_1}(x_1, y_1, z_1)$ in the above equation, we find that for $j=2, \ldots,  k,$
\begin{align*}
\frac{\partial c_{m_1, n_1}(x_2, y_2, z_2, \ldots, x_k, y_k, z_k)} {\partial {z_j}}
=\frac{\partial ^2c_{m_1, n_1}(x_2, y_2, z_2, \ldots, x_k, y_k, z_k)}{\partial {x_j}\partial {y_j}}.
\end{align*}
Thus by the inductive hypothesis, there exists a sequence $\lambda_{m_1, n_1, \ldots, m_k, n_k}$ independent of
$x_2, y_2, z_2, \ldots, x_k, y_k, z_k$ (of course independent of $x_1, y_1$ and $z_1$) such that
\begin{align*}
&c_{m_1, n_1}(x_2, y_2, z_2 \ldots, x_k, y_k, z_k)\\
&=\sum_{m_1, n_1, \ldots, m_k,  n_k=0}^\infty \lambda_{m_1, n_1, \ldots, m_k,  n_k}
H_{m_2, n_2}( x_2, y_2, z_2)\ldots H_{m_k, n_k}(x_k, y_k, z_k).
\end{align*}
Substituting this equation  into (\ref{liu:eqn2}),
we find that $f$ can be expanded into a  $k$-fold complex Hermite series. Conversely, if $f$ is a $k$-fold complex Hermite series, then it satisfies
the partial differential equations in Theorem~\ref{LiutriHermite:eqna} by using Proposition~\ref{triH:propb}. Hence we complete the proof of the theorem.
\end{proof}
To determine if a given function is an analytic functions in several complex variables,
we can use the following theorem due to Hartogs (see, for example, \cite[p. 28]{Taylor}).
\begin{thm}\label{hartogthm}
If a complex valued function $f(z_1, z_2, \ldots, z_n)$ is holomorphic (analytic) in each variable separately in a domain $U\in\mathbb{C}^n,$
then,  it is holomorphic (analytic) in $U.$
\end{thm}
 %/////////////////////////////////////////////////////////////////////
%%%%%%%%%%%%%%%%%%%%%%%%%%%%%%%%%%%%%%%%%%%%%%%%%%%%%%%%%%%%%%%%%%%%%%%%%%%%%%%%%%%%%%%%%%%%%%%%%
\section{The Poisson Kernel for the  complex Hermite polynomials }
%%%%%%%%%%%%%%%%%%%%%%%%%%%%%%%%%%%%%%%%%%%%%%%%%%%%%%%%%%%%%%%%%%%%%%%%%%%%%%%%%%%%%%%%%%%%%%%%%
In this section we will use Theorem~\ref{LiutriHermite:eqna} to give a completely new proof of  the following Poisson kernel for the complex Hermite polynomials.
This formula was first derived by Carlitz \cite[p.13]{carlitz1978} in 1978, and rediscovered by
W\"{u}nsche \cite{Wunsche1999} without proof in 1999.
Ismail \cite[Theorem~3.3]{Ismail2016} recovered it
as a specific case of his Kibble--Slepian formula. For other different proofs, please see \cite[Theorem~4.1]{Ghanmi2017}, \cite{IsmailZhang}. Our proof is brand new.
\begin{thm}\label{mehlerthm} For $|stz_1z_2|<1,$ the Mehler formula for the  complex Hermite polynomials states that
\begin{align*}
&\sum_{m, n=0}^\infty \frac{H_{m, n}(x_1, y_1, z_1)H_{m, n}(x_2, y_2, z_2)}{m! n!} {s^mt^n}\\
&=\frac{1}{1-stz_1z_2}
\exp \left(\frac{sx_1x_2+ty_1y_2+(z_1x_2y_2+z_2x_1y_1)st}{1-stz_1z_2}\right).
\end{align*}
\end{thm}
\begin{proof} If we use $f(x_1, y_1, z_1)$ to denote the right-hand side of the equation in
Theorem~\ref{mehlerthm}, then, it is easily seen that $f(x_1, y_1, z_1)$ is an analytic function of $x_1, y_1, z_1$ for any $x_1, y_1$ and $|stz_1z_2|<1.$ Hence $f(x_1, y_1, z_1)$ is analytic at $(x_1, y_1, z_1)=(0, 0, 0).$
By a direct computation, we find that
\[
\frac{\partial f}{\partial z_1}=\frac{\partial^2 f}{\partial x_1 \partial y_1}
=\left(\frac{z_2st}{(1-stz_1z_2)^2}+\frac{st(x_2+y_1z_2)(y_2+x_1z_2)}{(1-stz_1z_2)^2} \right)f.
\]
Thus, by Theorem~\ref{LiutriHermite:eqna},  there exists a sequence $\{\lambda_{m, n}\}$ independent of $x_1, y_1$ and $z_1$ such that
\begin{align}
&\frac{1}{1-stz_1z_2}
\exp \left(\frac{sx_1x_2+ty_1y_2+(z_1x_2y_2+z_2x_1y_1)st}{1-stz_1z_2}\right)\label{mehler:eqn1}\\
&=\sum_{m, n=0}^{\infty} \lambda_{m, n} H_{m, n}(x_1, y_1, z_1).\nonumber
\end{align}
Setting $z_1=0$ in this equation and using $H_{m, n}(x_1, y_1,  0)=x_1^m y_1^n$ in the resulting equation,  we immediately find that
\[
\exp (sx_1x_2+ty_1y_2+x_1y_1z_2st)
=\sum_{m, n=0}^{\infty} \lambda_{m, n} x_1^m y_1^n.
\]
Using the generating function for the  complex Hermite polynomials in Proposition~\ref{triH:propa}, we have
\[
\exp (sx_1x_2+ty_1y_2+x_1y_1z_2st)=\sum_{m, n=0}^\infty
\frac{H_{m, n}(x_2, y_2, z_2)}{m! n!} {(sx_1)^{m}(ty_1)^n}.
\]
Comparing the right-hand sides of these two equations, we conclude that
\[
\lambda_{m, n}=\frac{H_{m, n}(x_2, y_2, z_2)}{m! n!}s^m t^n.
\]
Substituting this into (\ref{mehler:eqn1}), we complete the proof of Theorem~\ref{mehlerthm}.
\end{proof}

Using Proposition~\ref{triH:propc},  we easily find that the Poisson kernel for the complex Hermite polynomials is equivalent to the following exponential operational identity, which is equivalent to \cite[Equation (5.1)]{Wunsche2015}.
\begin{thm}\label{Mehleroperator} For $|stz_1z_2|<1,$ we have the exponential operator identity
\begin{align*}
&\exp\left( z_2 \frac{\partial^2}{\partial x_2 \partial y_2}\right)
\left\{ \exp (sx_1x_2+ty_1y_2+ty_1y_2+stz_1x_2y_2) \right\}\\
&=\frac{1}{1-stz_1z_2}
\exp \left(\frac{sx_1x_2+ty_1y_2+(z_1x_2y_2+z_2x_1y_1)st}{1-stz_1z_2}\right).
\end{align*}
\end{thm}
%%%%%%%%%%%%%%%%%%%%%%%%%%%%%%%%%%%%%%%%%%%%%%%%%
%Theorem~\ref{mehlerthm} we can easily derive the following series expansion %formula.
%\begin{prop}\label{Mehlerapp:eqn1} For $|tx|<1, $ we have
%\[
%\sum_{n=0}^{\infty} \left(\sum_{k=0}^n {n\choose k}^2 k! x^{n-k} z^k\right) %\frac{t^n}{n!}=\frac{1}{1-tz} \exp\left(\frac{tx}{1-tz}\right).
%\]
%\end{prop}
%\begin{proof}
%Setting $x_2=y_2=0$ in Theorem~\ref{mehlerthm} and using the fact that
%$H_{m, n}(0, 0, z_2)=n!\delta_{m,n} z_2^n,$ we immediately find that
%\[
%\sum_{n=0}^\infty H_{n, n} (x_1, y_1, z_1) \frac{(tsz_2)^n}{n!}
%=\frac{1}{1-tsz_1z_2} \exp\left(\frac{tsx_1y_1z_2}{1-stz_1z_2}\right).
%\]
%Putting $s=z_2=1,$ replacing $x_1y_1$ by $x$ and $z_1$ by $z$, we complete %the proof of the proposition.
%\end{proof}
%%%%%%%%%%%%%%%%%%%%%%%%%%%%%%%%%%%%%%%%%%%%%%%%%%%%%%%%%%%%%%%%%%%%%%%%%%%%%
\section{The Nielsen type formulas for the complex Hermite polynomials}
%%%%%%%%%%%%%%%%%%%%%%%%%%%%%%%%%%%%%%%%%%%%%%%%%%%%%%%%%%%%%%%%%%%%%%%%%%%%%%
\noindent  We begin this section with the following formula for the  complex  Hermite polynomials.
\begin{thm}\label{Nielsen:thma} For any complex numbers $x, y, z, s_1, s_2, t_1$ and $t_2$, we have
\begin{align*}
&\exp\left((s_1+s_2)x+(t_1+t_2)y+(s_1+s_2)(t_1+t_2)z\right)\\
&=\sum_{m_1, n_1, m_2, n_2=0}^\infty H_{m_1+m_2, n_1+n_2} (x, y, z)
\frac{s_1^{m_1} s_2^{m_2}t_1^{n_1}t_2^{n_2}}{m_1! m_2! n_1! n_2!}.
\end{align*}
\end{thm}
\begin{proof} Denote the left-hand side of the equation in Theorem~\ref{Nielsen:thma} by $f(x, y, z)$. It is easily seen that
$f(x, y, z)$ is analytic at $(0, 0, 0)$. A simple computation shows that
\[
\frac{\partial f}{\partial z}
=\frac{\partial^2 f}{\partial  x \partial  y}
=(s_1+s_2)(t_1+t_2) f(x, y, z).
\]
Thus, by Theorem~\ref{LiutriHermite:eqna},  there exists a sequence $\{\lambda_{k, l}\}$ independent of $x, y$ and $z$ such that
\begin{align}
&\exp\left((s_1+s_2)x+(t_1+t_2)y+(s_1+s_2)(t_1+t_2)z\right)\label{Tnielsen:eqn1}\\
&=\sum_{k_1, l=0}^\infty \lambda_{k, l} H_{k, l}(x, y, z).\nonumber
\end{align}
Upon setting $z=0$ in the equation and using $H_{k, l}(x, y, 0)=x^ky^l,$
we deduce that
\[
\exp\left((s_1+s_2)x+(t_1+t_2)y\right)
=\sum_{k_1, l=0}^\infty \lambda_{k, l} x^k y^l.
\]
Equating the coefficients of $x^ky^l$ on both sides of this equation, we find that $k! l! \lambda_{k, l}=(s_1+s_2)^k(t_1+t_2)^l.$  Substituting this into the right-hand side of (\ref{Tnielsen:eqn1}), expanding $(s_1+s_2)^k (t_1+t_2)^l$ using the binomial theorem and interchanging the order of summation, we complete the proof of Theorem~\ref{Nielsen:thma}.
\end{proof}	
 Using Theorem~\ref{Nielsen:thma} and method of equating the coefficients of like power, we can derive the following Nielsen type formula for the  complex Hermite polynomials, which is equivalent to \cite[Equation (3.11)]{Ghanmi2013} and \cite[Equation (4.7)]{Ismail2016}.
\begin{thm}\label{Nielsen:thmb} For any non-negative integers $m_j, n_j, p_j$
$j\in\{1, 2\}$, we have	
\begin{align*}
&\frac{H_{m_1+m_2, n_1+n_2}(x, y, z)}{m_1! m_2! n_1! n_2!}\\
&=\sum_{p_1=0}^{m_1\land n_2}\sum_{p_2=0}^{n_1\land m_2}
\frac{H_{m_1-p_1, n_1-p_2}(x, y, z)H_{m_2-p_2, n_2-p_1}(x, y, z) z^{p_1+p_2}}{p_1!p_2!(m_1-p_1)!(m_2-p_2)!(n_1-p_2)!(n_2-p_1)!}.
\end{align*}
\end{thm}
%%%%%%%%%%%%%%%%%%%%%%%%%%%%%%%%%%%%%%%%%%%%%%%%
Upon multiplying both sides of the equation in Theorem~\ref{Nielsen:thma} by
$\exp(-s_1t_2-s_2t_1)z$ and then equating the coefficients of like power, we  can also derive the following formula due to Ismail \cite[Theorem~4.1]{Ismail2016}.
\begin{thm}\label{Nielsen:thmc} For any non-negative integers $m_j, n_j, p_j$
	$j\in\{1, 2\}$, we have	
	\begin{align*}
	&\frac{H_{m_1, n_1}(x, y, z)H_{m_2, n_2}(x, y, z)}{m_1! m_2! n_1! n_2!}\\
	&=\sum_{p_1=0}^{m_1\land n_2}\sum_{p_2=0}^{n_1\land m_2}
	\frac{H_{m_1+m_2-p_1-p_2, n_1+n_2-p_1-p_2}(x, y, z) (-z)^{p_1+p_2}}{p_1!p_2!(m_1-p_1)!(m_2-p_2)!(n_1-p_1)!(n_2-p_2)!}.
	\end{align*}
\end{thm}

%%%%%%%%%%%%%%%%%%%%%%%%%%%%%%%%%%%%%%%%%%%%%%%%%%%%%%%%%%%%%%%%%%%%%%%%%%%%%%%%%%%%%%%%%%%%%%%%%
\section{ Addition formula for the complex Hermite polynomials }
%%%%%%%%%%%%%%%%%%%%%%%%%%%%%%%%%%%%%%%%%%%%%%%%%%%%%%%%%%%%%%%%%%%%%%%%%%%%%

\begin{thm}\label{additionthm} If $M, N$ are two non-negative integers, then,
we have the following addition formula for the complex Hermite polynomials:	
\begin{align*}	
&H_{M, N} (a_1x_1+\cdots+a_kx_k, b_1y_1+\cdots+b_ky_k, a_1b_1z_1+\cdots+a_kb_kz_k)\\
&=\sum_{m_1, n_1, \ldots, m_k, n_k} \frac{M!N!}{m_1!n_1!\ldots m_k!n_k!}
a_1^{m_1}\cdots a_k^{m_k} b_1^{n_1}\cdots b_k^{n_k}\\
&\qquad \qquad \qquad \qquad \times H_{m_1, n_1}(x_1, y_1, z_1)\cdots H_{m_k, n_k}(x_k, y_k, z_k).
\end{align*}
The sum is taken over all combinations of non-negative integers indices $m_1$ through $m_k$ and $n_1$ through $n_k$  such that
\[
m_1+\cdots+m_k=M,~\text{and}~n_1+\cdots+n_k=N.
\]
\end{thm}
\begin{proof} Upon denoting the left-hand side of the equation in Theorem~\ref{additionthm} by
\[
f(x_1, y_1, z_1, \ldots, x_k, y_k, z_k),
\]
it is obvious that this function is analytic at $(0, \ldots, 0)\in \mathbb{C}^{3k}.$
For simplicity, we temporarily denote
\begin{align*}
x&=a_1x_1+\cdots+a_kx_k,\\
y&=b_1y_1+\cdots+b_ky_k,\\
z&=a_1b_1z_1+\cdots+a_kb_kz_k.
\end{align*}
By a simple calculation, we find that for $j=1, \ldots, k,$
\[
\frac{\partial f}{\partial z_j} =\frac{\partial^2 f}{\partial x_j \partial y_j}=a_jb_j \frac{\partial H_{M,N}}{\partial z}.
\]
Thus, by Theorem~\ref{LiutriHermite:eqna},  there exists a sequence $\{\lambda_{m_1, n_1, \ldots, m_k, n_k}\}$ independent of
\[
x_1, y_1, z_1,
\ldots, x_k, y_k, z_k
\]
  such that
\begin{align*}	
&H_{M, N} (a_1x_1+\cdots+a_kx_k, b_1y_1+\cdots+b_ky_k, a_1b_1z_1+\cdots+a_kb_kz_k)\\
&=\sum_{m_1, n_1, \ldots, m_k, n_k=0}^\infty  \lambda_{m_1, n_1, \ldots, m_k, n_k} H_{m_1, n_1}(x_1, y_1, z_1)\cdots H_{m_k, n_k}(x_k, y_k, z_k).
\end{align*}
Setting $z_1=\cdots=z_k=0$ and in the resulting equation using the fact that
\[
H_{m_j, n_j} (x_j, y_j, 0)=x_j^{m_j}y_j^{n_j},
\]
we deduce that
\begin{align*}
&(a_1x_1+\cdots+a_kx_k)^M (b_1y_1+\cdots+b_ky_k)^N\\
&=\sum_{m_1, n_1, \ldots, m_k, n_k=0}^\infty  \lambda_{m_1, n_1, \ldots, m_k, n_k} x_1^{m_1} y_1^{n_1}\cdots x_k^{m_k} y_k^{n_k}.
\end{align*}
Expanding the left-hand side by the multinomial theorem and then equating the coefficients of multiple power series, we complete the proof of Theorem~\ref{additionthm}.
\end{proof}	
%%%%%%%%%%%%%%%%%%%%%%%%%%%%%%%%%%%%%%%%%%%%%%%%%%%%%%%%%%%%%%%%%%%%%%%%%%%%%%
\section{A multilinear generating function for the complex Hermite polynomials }
%%%%%%%%%%%%%%%%%%%%%%%%%%%%%%%%%%%%%%%%%%%%%%%%%%%%%%%%%%%%%%%%%%%%%%%%%%%%%%
\begin{thm}\label{multipl:mehler} If $|s_1t_1z_1+\cdots+s_rt_rz_r|<1$ and
$a, b, c$ are defined by
\begin{align*}	
a&=s_1x_1+\cdots+s_rx_r,\\
b&=t_1y_1+\cdots+t_ry_r,\\
c&=s_1t_1z_1+\cdots+s_rt_rz_r,
\end{align*}
then, we have the following multilinear generating function for the complex
Hermite polynomials:
\begin{equation}
\frac{1}{(1-cz)} \exp \left(\frac{ax+by+cxy+abz}{1-cz}\right)
\label{mul:eqn1}
\end{equation}
\begin{align}
&=\sum_{m_1, n_1, \ldots, m_r, n_r=0}^\infty   H_{m_1, n_1}(x_1, y_1, z_1)\cdots H_{m_r, n_r}(x_r, y_r, z_r)\nonumber\\
&\qquad \qquad \qquad \qquad \qquad \times
H_{m_1+\cdots+m_r, n_1+\cdots+n_r}(x, y, z) \frac{s_1^{m_1}t^{n_1}\cdots s_r^{m_r}t_r^{n_r}}{m_1! n_1!\cdots m_r! n_r!}.\nonumber
\end{align}
\end{thm}
\begin{proof}
If we use  $f(x, y, z)$ to denote the left-hand side of (\ref{mul:eqn1}), then, it is easily seen that  $f$ is an analytic function of
$x, y, z$ such that  $|s_1t_1z_1+\cdots+s_rt_rz_r|<1.$ Hence $f(x, y, z)$ is analytic at $(x, y, z)=(0, 0, 0)$. By a straightforward computation, we conclude that
\[
\frac{\partial f}{\partial z} =\frac{\partial^2 f}{\partial x \partial y}= \left(\frac{c}{1-cz}+\frac{(a+cy)(b+cz)}{(1-cz)^2}\right)f.
\]
Thus, by Theorem~\ref{LiutriHermite:eqna},  there exists a sequence $\lambda_{k, l}$ independent of $x, y, z$  such that
\begin{equation}
f(x, y, z)=\sum_{k, l=0}^\infty \lambda_{k, l} H_{k, l}(x, y, z).
\label{mul:eqn2}
\end{equation}
Setting $z=0$ in the above equation and using the fact that $H_{k, l}(x, y, 0)=x^ky^l,$ we find that
\begin{equation}
f(x, y, 0)=\sum_{k_1, l=0}^\infty \lambda_{k, l} x^k y^l.
\label{mul:eqn3}
\end{equation}
On other hand, from the definition of $f(x, y, z)$, it is easily seen  that
\[
f(x, y, 0)=\prod_{j=1}^r \exp (s_jx_jx+t_jy_jy+s_jt_jz_jxy).
\]
Using the generating function of the exponential type for the complex Hermite polynomials in Proposition~\ref{triH:propa}, we find that
\begin{align*}
f(x, y, 0)&=\sum_{m_1, n_1, \ldots, m_r, n_r=0}^\infty   H_{m_1, n_1}(x_1, y_1, z_1)\cdots H_{m_r, n_r}(x_r, y_r, z_r)\\
&\qquad \qquad \qquad \qquad \times
 \frac{(s_1x)^{m_1}(t_1y)^{n_1}\cdots (s_rx)^{m_r}(t_ry)^{n_r}}{m_1! n_1!\cdots m_r! n_r!}.
\end{align*}
Comparing this equation with (\ref{mul:eqn3}) and equating the coefficients of $x^k y^l$, we conclude that
\begin{align*}
\lambda_{k, l}&=\sum_{{m_1+\cdots+m_r=k}\atop{n_1+\cdots+n_r=l}}^\infty   H_{m_1, n_1}(x_1, y_1, z_1)\cdots H_{m_r, n_r}(x_r, y_r, z_r)\\
&\qquad \qquad \qquad \qquad \times
\frac{s_1^{m_1}t_1^{n_1}\cdots s_r^{m_r}t_r^{n_r}}{m_1! n_1!\cdots m_r! n_r!}.
\end{align*}
Substituting this into (\ref{mul:eqn2}), we complete the proof of Theorem~\ref{multipl:mehler}.
\end{proof}

%%%%%%%%%%%%%%%%%%%%%%%%%%%%%%%%%%%%%%%%%%%%%%%%%%%%%%%%%%%%%%%%%%%%%%%%%%%%%%
\section{A generating function for the products of the Hermite polynomials and the complex Hermite polynomials}
%%%%%%%%%%%%%%%%%%%%%%%%%%%%%%%%%%%%%%%%%%%%%%%%%%%%%%%%%%%%%%%%%%%%%%%%%%%%%%
%%%%
\noindent
As usual, for any real number $x$, we use $[x]$ to denote the greatest integer function. For any complex number $x$, the Hermite polynomials are defined by
\begin{equation}
H_n(x)=\sum_{k=0}^{[\frac{n}{2}]}  \frac{n!}{k!(n-2k)!} (2x)^{n-2k}.
\label{Hermite:eqn1}
\end{equation}
 The exponential generating function for the  Hermite polynomials $H_n(x)$ is given by
\begin{equation}
\exp (2xt-t^2)=\sum_{n=0}^\infty \frac{H_n(x)}{n!} t^n, \quad |t|<\infty.
	\label{Hermite:eqn2}
\end{equation}

The following formula  is equivalent to W\"{u}nsche \cite[Equation (7.4)]{Wunsche2015}. In his paper  Professor W\"{u}nsche just said that his formula can be proved by using auxiliary formulae prepared in Appendix A, but lacks sufficient details. Now we will use Theorem~\ref{LiutriHermite:eqna} to give a very  simple proof of Theorem~\ref{mixed:mehler}.
\begin{thm}\label{mixed:mehler} For $|2stz|<1,$ we have the following generating function for the Hermite polynomials and the complex Hermite polynomials.	
\begin{align*}	
&\sum_{m, n=0}^\infty (-1)^{m+n} H_{m, n}(x, y, z) H_m(u)H_n(v)\frac{s^mt^n}{m!n!}\\
&=\frac{\exp(u^2+v^2)}{\sqrt{1-4s^2t^2z^2}}
\exp\left(\frac{4stz(sx+u)(ty+v)-(sx+u)^2-(ty+v)^2}{1-4s^2t^2z^2}\right).
\end{align*}
\end{thm}

\begin{proof} If we use $f(x, y, z)$ to denote the right-hand side of the equation in Theorem~\ref{mixed:mehler}, then, it is easily seen that $f$ is analytic at $(0, 0, 0)$. A elementary calculation shows that
\begin{align*}
&\frac{\partial f}{\partial z} =\frac{\partial^2 f}{\partial x \partial y}=\\
& \left\{\frac{4s^2t^2z}{1-4s^2t^2z^2}
+\frac{4st(2stz(sx+u)-(yv+t))(2stz(ty+v)-(sx+u))}{(1-4s^2t^2z^2)^2} \right\}f
\end{align*}
Hence, by Theorem~\ref{LiutriHermite:eqna},  there exists a sequence $\lambda_{m, n}$ independent of $x, y, z$  such that
\begin{equation}
f(x, y, z)=\sum_{m, n=0}^\infty \lambda_{m, n} H_{m, n}(x, y, z).
\label{mix:eqn1}
\end{equation}
Setting $z=0$ in the above equation and using the fact that $H_{m, n}(x, y, 0)=x^my^n,$ we deduce that
\begin{equation}
\exp(-(sx)^2-2sxu-(ty)^2-2tyv)=\sum_{m, n=0}^\infty \lambda_{m, n} x^m y^n.
\label{mix:eqn2}
\end{equation}
Using the exponential  generating function for the Hermite polynomials, we find that
\begin{align*}
\exp(-(sx)^2-2sxu)=\sum_{m=0}^\infty H_m(u)\frac{(-sx)^m}{m!},\\
\exp(-(ty)^2-2tyv)=\sum_{n=0}^\infty H_n(v) \frac{(-ty)^n}{n!}.
\end{align*}
Upon substituting these two equations into the left-hand side of (\ref{mix:eqn2}) and equating the coefficients of like power, we obtain
\[
\lambda_{m, n}=(-1)^{m+n}H_m(u)H_n(v) \frac{s^m t^n}{m! n!}.
\]
Combining this equation with (\ref{mix:eqn1}), we complete the proof of Theorem~\ref{mixed:mehler}.
\end{proof}
Theorem~\ref{mixed:mehler} contains the Mehler formula for the Hermite polynomials as a special case, which was  discovered by Mehler \cite[p.174, Equation(18)]{MehlerFG1866} in 1866. One can also find this important formula in most books on special functions, for example, \cite[p.280, Equation (6.1.13)]{AndAskRoy1999},   \cite[p.111, Equation(4.417)]{BealsWong2010}, \cite[p.108, Equation (4.7.6)]{Ismail2009}, \cite[p. 198, Equation (2)]{Rainville1960}. One very simple proof of this formula can be found in
\cite{LiuHermite2017}.
\begin{thm}\label{mixed:mehlera} For $|2t|<1,$ we have the Mehler formula for the Hermite polynomials:
\[
\sum_{n=0}^{\infty} \frac{H_n(u)H_n(v)}{n!}t^n
=\frac{1}{\sqrt{1-4t^2}}
\exp\left(\frac{4tuv-4(u^2+v^2)t^2}{1-4t^2}\right).
\]	
\end{thm}
\begin{proof}
Upon taking $x=y=0$ in the equation in Theorem~\ref{mixed:mehler} and using the fact that
\[
H_{m, n} (0, 0, z)=\delta_{m,n} n! z^n,
\]
in the resulting equation,  we immediately  conclude  that
\[
\sum_{n=0}^{\infty} \frac{H_n(u)H_n(v)}{n!}(stz)^n
=\frac{\exp(u^2+v^2)}{\sqrt{1-4s^2t^2z^2}}
\exp\left(\frac{4stuv-(u^2+v^2)}{1-4s^2t^2z^2}\right).
\]
Putting $s=z=1$ in this equation and simplifying we complete the proof of
Theorem~\ref{mixed:mehlera}.
\end{proof}
In the same way we can prove the following more general generating function formula, which appeared to be new.
\begin{thm}\label{mixed:mehlerb} If $k$ is a non-negative integer and  $|2stz|<1,$ we have the following generating function for the Hermite polynomials and the complex Hermite polynomials:
	\begin{align*}	
	&\sum_{m, n=0}^\infty (-1)^{m+n} H_{m, n}(x, y, z) H_{m+k}(u)H_n(v)\frac{s^mt^n}{m!n!}\\
	&=\frac{\exp(u^2+v^2)}{(1-4s^2t^2z^2)^{(k+1)/2}}
	H_k\left(\frac{u+sx-2stz(v+ty)}{\sqrt{1-4s^2t^2z^2}}\right)\\
&\qquad \times 	\exp\left(\frac{4stz(sx+u)(ty+v)-(sx+u)^2-(ty+v)^2}{1-4s^2t^2z^2}\right).
	\end{align*}
\end{thm}
Upon putting $x=y=0$ in Theorem~\ref{mixed:mehlerb}  and using the fact that
\[
H_{m, n} (0, 0, z)=\delta_{m,n} n! z^n,
\]
in the resulting equation and finally setting $s=z=1,$ we derive the following formula due to Weisner \cite[Equation (4.9)]{Weisner1959}.
\begin{thm}\label{mixed:mehlerd}
For $|2t|<1,$ we have 	
\begin{align*}
&\sum_{n=0}^{\infty} \frac{H_{n+k}(u)H_n(v)}{n!}t^n\\
&=\frac{1}{(1-4t^2)^{(k+1)/2}} H_k \left( \frac{u-2tv}{\sqrt{1-4t^2}}\right)
\exp\left(\frac{4tuv-4(u^2+v^2)t^2}{1-4t^2}\right).
\end{align*}
\end{thm}

\section{Acknowledgments}
The author is  grateful to  the editor and the referees for their valuable comments and suggestions.


\begin{thebibliography}{9}

\bibitem{AliBH2019}	
	
S.T. Ali, F. Bagarello, G. Honnouvo, Modular structures on trace class operators and applications to Landau levels, J. Phys. A 43 (2010) 105--202.


\bibitem{AndAskRoy1999}
G. E. Andrews, R. Askey and R. Roy, Special Functions, Cambridge University Press, Cambridge, 1999.	

\bibitem{BealsWong2010} R. Beals and  R. Wong,
Special Functions, Cambridge University Press, Cambridge, 2010.
	
\bibitem{carlitz1978}	
	
L. Carlitz, A set of polynomials in three variables, Houston J. Math. 4 (1) (1978) 11--33.

\bibitem{DattoliOTTOVA1997}
G. Dattoli, P.L. Ottaviani, A. Torre, L. Vazquez, Evolution operator equations, integration with algebraic and  finite difference methods: applications to physical problems in classical and quantum mechanics, Riv. Nuovo Cimento 20 (1997) 1--133.	
		
\bibitem{Ghanmi2013}		
A. Ghanmi, Operational formulae for the complex Hermite polynomials $H_{p, q}(z, \bar{z})$, Integral Transforms Spec. Funct. 24 (2013) 884--895.

\bibitem{Ghanmi2017}
A. Ghanmi, Mehler's formulas for the univariate complex Hermite polynomials and applications. Math. Methods Appl. Sci. 40 (2017) 7540--7545.


\bibitem{Ismail2009}M. E. H. Ismail,
Classical and quantum orthogonal polynomials in one variable, Encyclopedia of Mathematics and its Applications, Vol. 98, Cambridge University Press, Cambridge, 2009.	
	
	
\bibitem{Ismail2016}M. E. H. Ismail,	
 Analytic properties of complex Hermite polynomials. Trans. Amer. Math. Soc. 368 (2016) 1189--1210.

\bibitem{IsmailZhang}
M.E.H. Ismail, R. Zhang,
Kibble-Slepian formula and generating functions for 2D polynomials. Adv. in Appl. Math. 80 (2016) 70--92.

\bibitem{IsmailZeng2015}

M. E. H. Ismail, J. Zeng, A two variable extension of the Laguerre and disc polynomials, J. Math. Anal. Appl. 424 (2015) 289--303.

\bibitem{Ito1952}
K. It\^{o}, Complex multiple Wiener integral, Jpn. J. Math. 22 (1952) 63--86.

\bibitem{LiuHermite2017}
Z.-G. Liu,  On a system of partial differential equations and the bivariate Hermite polynomials, J. Math. Anal. Appl. 454 (2017) 1--17.	


\bibitem{Malgrange} B. Malgrange,
Lectures on functions of several complex variables, Springer-Verlag, Berlin, 1984.

\bibitem{MehlerFG1866} F. G. Mehler,
Ueber die Entwicklung einer Function von beliebig vielen Variabeln nach Laplaceschen Functionen h\"{o}herer Ordnung,
J. Reine Angew. Math.   66 (1866) 161--176.


\bibitem{Rainville1960}  E. D. Rainville,
Special Functions, The Macmillan Company, New York 1960.


\bibitem{Taylor}J. Taylor,
Several Complex Variables with Connections to Algebraic Geometry and Lie Groups,
Graduate Studies in Mathematics, vol. 46. Am. Math. Soc., Providence,  2002.


\bibitem{Weisner1959} L. Weisner,
Generating functions for Hermite functions, Canad. J. Math. 11 (1959) 141--147.

\bibitem{Wunsche1998}A. W\"{u}nsche,
 Laguerre 2D-functions and their application in quantum optics, J. Phys. A 31 (1998) 8267--8287.

\bibitem{Wunsche1999}
A. W\"{u}nsche, Transformations of Laguerre 2D-polynomials and their applications to quasiprobabilities, J. Phys. A 21 (1999) 3179--3199.

\bibitem{Wunsche2015} A. W\"{u}nsche, Generating functions for products of special Laguerre 2D and Hermite 2D polynomials, Applied Mathematics 6 (2015) 2142--2168.

\end{thebibliography}
\end{document}